\documentclass[10pt,twoside]{amsart}
\usepackage{geometry}
\usepackage[english]{babel}
\usepackage{graphicx}
\usepackage{amsmath}
\usepackage{amsfonts}

\begin{document}

\newcommand{\wk}{\mbox{$\,<$\hspace{-5pt}\footnotesize )$\,$}}

\numberwithin{equation}{section}
\newtheorem{teo}{Theorem}
\newtheorem{lemma}{Lemma}
\newtheorem{defi}{Definition}
\newtheorem{coro}{Corollary}
\newtheorem{prop}{Proposition}
\theoremstyle{remark}
\newtheorem{remark}{Remark}
\newtheorem{scho}{Scholium}
\numberwithin{lemma}{section}
\numberwithin{prop}{section}
\numberwithin{teo}{section}
\numberwithin{defi}{section}
\numberwithin{coro}{section}
\numberwithin{figure}{section}
\numberwithin{remark}{section}
\numberwithin{scho}{section}

\title{A new construction of Radon curves and related topics}

\author{Vitor Balestro, Horst Martini, and Ralph Teixeira}
\address [V. Balestro] {CEFET/RJ Campus Nova Friburgo - Nova Friburgo - Brazil;
\newline
\indent Instituto de Matem\'{a}tica e Estat\'{i}stica - UFF - Niter\'{o}i - Brazil}
\email{vitorbalestro@mat.uff.br}
\address [H. Martini] {Fakult\"at f\"ur Mathematik - Technische Universit\"at Chemnitz - 09107 Chemnitz - Germany;
\newline
\indent Dept. of Applied Mathematics - Harbin University of Science and Technology - 150080 Harbin - China}
\email{martini@mathematik.tu-chemnitz.de}
\address [R. Teixeira] {Instituto de Matem\'{a}tica e Estat\'{i}stica - UFF - Niter\'{o}i - Brazil}
\email{ralph@mat.uff.br}

\begin{abstract} We present a new construction of Radon curves which only uses convexity methods. In other words, it does not rely on an auxiliary Euclidean background metric (as in the classical works of J. Radon, W. Blaschke, G. Birkhoff, and M. M. Day), and also it does not use typical methods from plane Minkowski Geometry (as proposed by H. Martini and K. J. Swanepoel). We also discuss some properties of normed planes whose unit circle is a Radon curve and give characterizations of Radon curves only in terms of Convex Geometry.
\end{abstract}

\maketitle

\section{Introduction}

      Continuing and completing the investigations from \cite{martiniantinorms}, we will study \textit{Radon curves} from a slightly new point of view. Recall that Radon curves are centrally symmetric, closed, convex curves in the plane with the following property: when they are chosen  as unit circle of a norm, then (and only then) Birkhoff orthogonality is symmetric. Such curves appeared first in Radon's paper \cite{radon} (see also \cite{Bla}), and later new constructions were given by Birkhoff and Day in \cite{birkhoff} and \cite{day}, respectively. All these constructions are given in terms of some auxiliary Euclidean background metric (considering usual polarity and a $90^o$ rotation). In \cite{martiniantinorms}, Martini and Swanepoel gave a construction which does not need an auxiliary Euclidean structure. The starting points are a curve positioned within a quadrant
     (which is determined by any fixed pair of linearly independent vectors in the plane), a ``norm" defined in this quadrant by the curve, and a determinant form to obtain a norm in the adjacent quadrants in such a way that the union of such curve pieces (together with the original piece reflected at the origin) form the unit circle of a Radon norm. What we propose here is like a change of this method: we construct Radon curves using only convexity methods, and after that we show their desired properties when such a curve is chosen as unit circle of a normed plane.

We shall fix some notation. Throughout the text, $V$ denotes a \textit{two-dimensional vector space} (whose \textit{origin} is denoted by $o$), and $[\cdot,\cdot]$ stands for a non-degenerate \textit{symplectic bilinear form} (a determinant) on it. We denote by $[ab]$, $\left<ab\right>$ and $\left.[ab\right>$ the \textit{closed segment} connecting $a$ and $b$, the \textit{line} spanned by $a$ and $b$, and the \textit{half-line} with origin $a$ and through $b$; $(ab)$ is the (relatively) \textit{open segment} from $a$ to $b$. A compact, convex set $K \subset V$ with interior points is called a \textit{convex body}; by $\partial K$ and int$K$ we denote the \textit{boundary} and the
\textit{interior} of $K$, respectively. The unit ball $B$ of a normed plane is always a convex body centred at the origin. 
When the plane $V$ is endowed with a norm $||\cdot||$, then $(V, \|\cdot\|)$ is called a \textit{normed} or \textit{Minkowski plane} with $B:= \{x \in V : \| x \| \le 1\}$ and $S:=\{x \in V : \|x\|=1\}$
  as \textit{unit ball} and \textit{unit circle}, respectively. We say that a vector $x$ is \textit{Birkhoff orthogonal} to a vector $y$ if $||x|| \leq ||x+ty||$ for every $t \in \mathbb{R}$; this is denoted by $x \dashv_B y$.

For basics from the geometry of normed spaces, called Minkowski Geometry, we refer the reader to the monograph \cite{thompson} and the
surveys \cite{martini1} and \cite{martini2}. For Radon norms the main reference is \cite{martiniantinorms}, and \cite{alonso} is a suitable survey on orthogonality concepts in normed spaces.

\section{Background results}\label{back}

Within this section we briefly outline some background results from Convex Geometry that will be needed later. In Minkowski Geometry, we say that two vectors $x$ and $y$ present \textit{conjugate directions} if $x \dashv_B y$ and $y \dashv_B x$. This is equivalent to say that $x$ and $y$ are the directions of the sides of a parallelogram circumscribed to the unit circle $S$ and touched by it in the midpoints of its sides. In terms of Convex Geometry, we may formulate this as follows.

\begin{lemma}\label{lemmap} Any centrally symmetric two-dimensional convex body $K$ has a circumscribed parallelogram which touches $\partial K$ in the midpoints of its sides, and the directions of the sides of such a parallelogram are called conjugate directions (regarding the convex body K).
\end{lemma}
\begin{proof} This follows immediately from the fact that every norm in a Minkowski plane admits a pair of conjugate directions (see \cite{martini1}, Proposition 39, for a proof).

\end{proof}

The next proposition is concerned with supporting lines of plane convex bodies. It states that the quadrants defined by conjugate directions are, in some sense, dual regarding supporting relations.

\begin{prop}\label{prop25} Assume that $K$ is a plane convex body which is symmetric with respect to the origin (by translation, if necessary). Let $P$ be a parallelogram circumscribed about $K$ which touches $\partial K$ in the midpoints of its sides, and assume that its sides are in the directions $v,w$, where we choose these points in $\partial K$; i.e., we set $v,w \in \partial K$. Denote by $Q_1$ and $Q_2$ the usual first and second (closed) quadrants determined by the system of coordinates $\{v,w\}$, and let $\partial K_1 = \partial K\cap Q_1$ and $\partial K_2 = \partial K\cap Q_2$. Then, given an arbitrary point $p\in \partial K_1 \setminus\{v,w\}$, the direction of any supporting line $l$ to $K$ through $p$ must lie in $Q_2$. Moreover, any direction of $Q_2$ supports $K$ at some point of $\partial K_1$. Clearly, the same holds if we interchange the indices.

\end{prop}
\begin{proof} This is a basic, elementary result from Convex Geometry (rather than from Minkowski Geometry), and we will not give the algebraic details. First, if $p \in \partial K_1\setminus\{v,w\}$, then any supporting line $l$ to $K$ through $p$ must lie in the double cone determined by the lines $\left<pw\right>$ and $\left<vp\right>$ which does not contain $v+w$, since otherwise $l$ would separate the points $v$ and $w$ (see Figure \ref{fig57}). It is straightforward that any direction within this cone is a direction of $Q_2$ (this follows from the fact that $p$ must be contained in the parallelogram $ow(v+w)v$). Now, let $y \in Q_2$. The directions $y = v$ or $y = w$ support $K$ at $w$ and $v$, respectively. If $y \in \mathrm{int}(Q_2)$, we use the simple fact that any direction supports a given convex body in at least two points. Since any line in the direction $y$ through a point from $\mathrm{int}(Q_2)$ separates $v$ and $-v$ or $w$ and $-w$, it follows that $y$ must support $K$ at some point of $Q_1$.

\begin{figure}[t]
\centering
\includegraphics{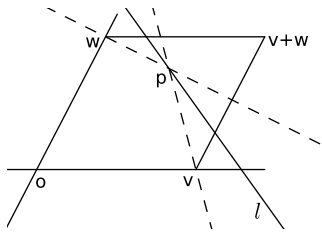}
\caption{Supporting line to $B$}
\label{fig57}
\end{figure}

\end{proof}

\section{Constructing Radon curves}\label{radon}

Let $V$ be a plane endowed with a non-degenerate symplectic form $[\cdot,\cdot]$ and fix linearly independent vectors $v,w \in V$ with $[v,w] = 1$. Consider the four usual quadrants $Q_1$, $Q_2$, $Q_3$, and $Q_4$ determined by the system of coordinates $\{v,w\}$, and let $\gamma_1$ be a curve connecting the points $v$ and $w$ with $\gamma_1 \subseteq \mathrm{conv}\{o,v,w,v+w\}$ such that the union of the segments $[ov]$ and $[ow]$ with $\gamma_1$ is the boundary of a convex body $K_1$, say. \\

We now define a (simple) curve $\gamma_2$ in the second quadrant as follows:
\begin{align*} \gamma_2(\lambda) = \frac{(1-\lambda) w + \lambda(-v)}{\sup_{x\in\gamma_1}|[x,(1-\lambda)w + \lambda(-v)]|}, \ \lambda \in [0,1]. \end{align*}
Notice that each point of the curve $\gamma_2$ is the image of a point $z_{\lambda} = (1-\lambda)w+\lambda(-v)$, $\lambda \in [0,1]$, from the segment $[(-v)w]$ by a homothety with center in the origin and ratio $1/s(\lambda)$, where $s:[0,1]\rightarrow\mathbb{R}$ is the function defined as $s(\lambda):= \sup_{x\in\gamma_1}|[x,z_{\lambda}]|$. Hence, in order to obtain geometric properties of $\gamma_2$, we will study this function.

\begin{lemma}\label{lemma20} For the function $s:[0,1] \rightarrow \mathbb{R}$ defined above we have \\

\normalfont\noindent\textbf{(i)} $\max\{1-\lambda,\lambda\} \leq s(\lambda) \leq 1$ \textit{for all $\lambda \in [0,1]$. In particular, $s(0) = s(1) = 1$}. \\

\noindent\textbf{(ii)} \textit{The function s is convex. In other words, if $0\leq\lambda_0\leq\lambda\leq\lambda_1\leq1$, then}
\begin{align*} (\lambda_1-\lambda_0)s(\lambda) \leq (\lambda_1-\lambda)s(\lambda_0)+(\lambda-\lambda_0)s(\lambda_1) \,.
\end{align*}
\end{lemma}
\begin{proof} We shall begin with \textbf{(i)}. Still writing $z_{\lambda} = (1-\lambda)w + \lambda(-v)$, we see that the inequality $\max\{1-\lambda,\lambda\} \leq s(\lambda)$ comes immediately from $|[v,z_{\lambda}]| = 1-\lambda$ and $|[w,z_{\lambda}]| = \lambda$. Now we prove $s(\lambda) \leq 1$. For any $x \in \gamma_1$ we have $x \in \mathrm{conv}\{o,v,w,v+w\}$, and hence $x$ can be written in the form $\alpha v + \beta w$ for some $\alpha,\beta \in [0,1]$. Thus,
\begin{align*}
|[x,z_{\lambda}]| = |[\alpha v + \beta w, (1-\lambda)w + \lambda(-v)]| = \alpha(1-\lambda) + \beta\lambda \leq 1.
\end{align*}
This yields the desired. \\

To prove \textbf{(ii)}, we clearly may assume $\lambda_0 \neq \lambda_1$ and write
\begin{align*} z_\lambda = \frac{\lambda_1-\lambda}{\lambda_1-\lambda_0}z_{\lambda_0} + \frac{\lambda - \lambda_0}{\lambda_1 - \lambda_0}z_{\lambda_1}. \end{align*}
\noindent Thus,
\begin{align*} s(\lambda) = \sup_{x\in\gamma_1}|[x,z_\lambda]| = \sup_{x\in\gamma_1}\left|\left[x,\frac{\lambda_1-\lambda}{\lambda_1-\lambda_0}z_{\lambda_0} + \frac{\lambda - \lambda_0}{\lambda_1 - \lambda_0}z_{\lambda_1}\right]\right| \leq \\ 
 \leq \frac{\lambda_1-\lambda}{\lambda_1-\lambda_0}\sup_{x\in\gamma_1}|[x,z_{\lambda_0}]| +  \frac{\lambda - \lambda_0}{\lambda_1 - \lambda_0}\sup_{x\in\gamma_1}|[x,z_{\lambda_1}]| = \frac{\lambda_1-\lambda}{\lambda_1-\lambda_0}s(\lambda_0) + \frac{\lambda - \lambda_0}{\lambda_1 - \lambda_0}s(\lambda_1), \end{align*}
and the proof is finished.

\end{proof}

\begin{coro}[Properties of $\gamma_2$]\label{coro16} The curve $\gamma_2$ constructed previously has, similarly to $\gamma_1$, the following properties: \\

\normalfont\noindent\textbf{(i)} $\gamma_2(0) = w$ \textit{and} $\gamma_2(1) = -v$, \\

\noindent\textbf{(ii)} $\gamma_2 \subseteq \mathrm{conv}\{-v,w,w-v\}$, \textit{and} \\

\noindent\textbf{(iii)} \textit{the union of $\gamma_2$ with the segments $[o(-v)]$ and $[ow]$ is the boundary of a convex body $K_2$, say.}
\end{coro}
\begin{proof} Assertion \textbf{(i)} follows immediately from $s(0) = s(1) = 1$. For \textbf{(ii)}, if we assume that $\lambda \in \left[0,\frac{1}{2}\right]$, then the ray $\left.[o\gamma_2(\lambda)\right>$ intersects the segments $[(-v)w]$ and $[w(w-v)]$ at the points $z_{\lambda} = (1-\lambda)w + \lambda(-v)$ and $y_{\lambda} = \frac{1}{1-\lambda}\left((1-\lambda)w+\lambda(-v)\right)$, respectively. Hence, since we have $1-\lambda \leq s(\lambda) \leq 1$, it follows that $1 \leq \frac{1}{s(\lambda)}\leq \frac{1}{1-\lambda}$. This gives $\gamma_2(\lambda) \in \left[z_{\lambda}y_{\lambda}\right]$. This last segment is obviously contained in the desired convex region.\\

\begin{figure}[t]
\centering
\includegraphics{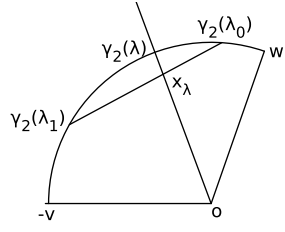}
\caption{Convexity of $\gamma_2$}
\label{fig56}
\end{figure}

To prove \textbf{(iii)} it is clearly enough to show that, for any $\lambda_0,\lambda_1 \in [0,1]$ and $\lambda_0 \leq \lambda \leq \lambda_1$, the intersection $x_{\lambda}$ of the ray $\left.[o\gamma_2(\lambda)\right>$ with the segment $[\gamma_2(\lambda_0)\gamma_2(\lambda_1)]$ obeys $x_{\lambda} \in \left.[\gamma_2(\lambda)o\right>$ (see Figure \ref{fig56}). Writing $x_{\lambda} = \alpha((1-\lambda)w + \lambda(-v))$ for some $\alpha \in \mathbb{R}$ (which is, by \textbf{(ii)}, $\geq 1$), we just have to prove that $\alpha \leq \frac{1}{s(\lambda)}$. Since there exists $\beta \in [0,1]$ such that $x_{\lambda} = (1-\beta)\gamma_2(\lambda_0) + \beta\gamma_2(\lambda_1)$,
we have the equalities
\begin{align*} \alpha(1-\lambda) = \frac{(1-\beta)(1-\lambda_0)}{s(\lambda_0)} + \frac{\beta(1-\lambda_1)}{s(\lambda_1)}, \ \mathrm{and} \end{align*}
\begin{align*} \alpha\lambda = \frac{(1-\beta)\lambda_0}{s(\lambda_0)} + \frac{\beta\lambda_1}{s(\lambda_1)}. \end{align*}
This can be seen as a system of equations in the variables $\alpha$ and $\beta$. Thus, we may calculate $\alpha$ in terms of $\lambda,\lambda_0,\lambda(1), s(\lambda_0)$, and $s(\lambda_1)$.
After some small calculation we have
\begin{align*}\alpha = \frac{\lambda_1-\lambda_0}{(\lambda_1-\lambda)s(\lambda_1) + (\lambda_0-\lambda)s(\lambda_0)}. \end{align*}
Hence $\alpha \leq \frac{1}{s(\lambda)}$ if and only if $(\lambda_1-\lambda_0)s(\lambda) \leq (\lambda_1-\lambda)s(\lambda_0)+(\lambda-\lambda_0)s(\lambda_1)$. But this is precisely item \textbf{(ii)} of the previous lemma.\\

\end{proof}

Since $\gamma_2$ connects $w$ and $-v$, it follows that the curve $\gamma = \gamma_1\cup\gamma_2\cup(-\gamma_1)\cup(-\gamma_2)$ is a closed, centrally symmetric curve. Curves constructed in this way are called \textit{Radon curves}. The next step is to prove that they form the boundaries of convex bodies.

\begin{prop} Any Radon curve is the boundary of a convex body.
\end{prop}
\begin{proof} We use here the same notation as above. A segment $[pq]$ connecting two points of $\gamma$ is, in particular, contained in the parallelogram $\mathrm{conv}\{w+v,w-v,-w-v,-w+v\}$, and therefore it can intersect the axes $\left<(-v)v\right>$ and $\left<(-w)w\right>$ only within the segments $[(-v)v]$ and $[(-w)w]$, respectively. Thus, considering these (possible) intersections, we may write $[pq]$ as a union of segments such that both endpoints of each of them belong to one of the convex bodies $K_1$, $K_2$, $-K_1$, or $-K_2$. Hence $[pq]$ is contained in the union of these sets, which is precisely the region enclosed by $\gamma$.

\end{proof}

\begin{coro}\label{coro17} Any direction of $Q_2$ supports $\mathrm{conv}(\gamma)$ at some point of $\gamma_1$. Furthermore, the direction of any supporting line to $\mathrm{conv}(\gamma)$ through a point of $\gamma_1\setminus\{v,w\}$ must lie in $Q_2$. The same holds if we interchange the indices.
\end{coro}
\begin{proof} By construction it is immediate that $v$ and $w$ are conjugate directions for the centrally symmetric convex set $\mathrm{conv}(\gamma)$. Hence we just have to apply Proposition \ref{prop25}.

\end{proof}

Now we prove a sort of duality that holds for Radon curves: if we start with $\gamma_2$ and define a curve in the first quadrant in the same way that we did it before, we would obtain precisely $\gamma_1$. This is presented by  the next lemma. But first we notice that, by convexity, a ray from the origin $o$ through a point of the segment $[wv]$ must intersect $\gamma_1$ in exactly one point. Hence we may parametrize $\gamma_1$ by $\gamma_1(\theta) = f(\theta)((1-\theta)w+\theta v)$ for some continuous function $f:[0,1] \rightarrow \mathbb{R}$. Observe that, in particular, $f \geq 1$.

\begin{lemma}\label{lemma21} The curve
\begin{align*}
\theta \mapsto \frac{(1-\theta)w+\theta v}{\sup_{y\in\gamma_2}|[y,(1-\theta)w + \theta v]|}, \ \theta\in [0,1] \,,
\end{align*}
coincides with $\gamma_1$. In other words, the function $f(\theta)$ defined above may be written in terms of $\gamma_2$ as $f(\theta) = \left(\sup_{y\in\gamma_2}|[y,(1-\theta)w+\theta v]|\right)^{-1}$.
\end{lemma}
\begin{proof} Let $\theta \in [0,1]$ be arbitrary and assume that
\begin{align*} \sup_{y\in\gamma_2}|[y,(1-\theta)w+\theta v]| = |[\gamma_2(\lambda),(1-\theta)w+\theta v]|, \end{align*}
i.e., the supremum is attained for the parameter $\lambda$ (considering the previously defined parametrization of $\gamma_2$). This yields
\begin{align*}
\sup_{y\in\gamma_2}|[y,(1-\theta)w+\theta v]| = \\ 
= \frac{|[(1-\lambda)w+\lambda(-v),(1-\theta)w+\theta v]|}{\sup_{x\in\gamma_1}|[x,(1-\lambda)w+\lambda(-v)]|} = \frac{\theta(1-\lambda)+\lambda(1-\theta)}{\sup_{x\in\gamma_1}|[x,(1-\lambda)w+\lambda(-v)]|} \leq \\ 
\leq \frac{\theta(1-\lambda)+\lambda(1-\theta)}{|[f(\theta)((1-\theta)w+\theta v),(1-\lambda)w + \lambda(-v)]|} = \frac{1}{f(\theta)}  \,.
\end{align*}
\noindent Hence $f(\theta) \leq \left(\sup_{y\in\gamma_2}|[y,(1-\theta)w+\theta v]|\right)^{-1}$. To prove the inverse inequality, we first notice that there exists a number $\sigma \in [0,1]$ such that
\begin{align*} \sup_{x\in\gamma_1}|[x,(1-\sigma)w+\sigma(-v)]| = |[\gamma_1(\theta),(1-\sigma)w+\sigma(-v)]|. \end{align*}
\noindent In fact, choose a line $l$ supporting $\mathrm{conv}(\gamma)$ and passing through $\gamma_1(\theta)$ whose direction lies in the second quadrant (the existence of such a line is guaranteed by Corollary \ref{coro16}). Hence $l$ is the line $t\mapsto\gamma_1(\theta)+t((1-\sigma)w+\sigma(-v))$ for some $\sigma \in [0,1]$. Thus, given any $\alpha \in [0,1]$, the ray $\left.[o\gamma_1(\alpha)\right>$ meets $l$ at a point $\gamma_1(\theta) + t_0((1-\sigma)w+\sigma(-v))$ for some $t_0\in\mathbb{R}$, and we get
\begin{align*} |[\gamma_1(\alpha),(1-\sigma)w+\sigma(-v)]| \leq \\ 
\leq |[\gamma_1(\theta)+t_0((1-\sigma)w+\sigma(-v)),(1-\sigma)w+\sigma(-v)]|
=\\ 
= |[\gamma_1(\theta),(1-\sigma)w+\sigma(-v)]|. \end{align*}
\noindent This shows the desired. Now,
\begin{align*} \sup_{y\in\gamma_2}|[y,(1-\theta)w+\theta v]| \geq |[\gamma_2(\sigma),(1-\theta)w+\theta v]| = \\
= \frac{|[(1-\sigma)w+\sigma(-v),(1-\theta)w+\theta v]|}{\sup_{x\in\gamma_1}|[x,(1-\sigma)w+\sigma v]|} = \frac{\theta(1-\sigma) + \sigma(1-\theta)}{|[\gamma_1(\theta),(1-\sigma)w+\sigma(-v)]|} = \\ 
=\frac{1}{f(\theta)},
\end{align*}
\noindent and this is the inverse inequality that we wanted. The proof is finished.

\end{proof}

In the next lemma, which is a technical one, we will explore a little better the assumption made (within the proof of the last lemma) on supporting lines with directions that realize the supremum of the determinant form.

\begin{lemma}\label{lemma23} Let $\lambda \in [0,1]$. The supremum $\sup_{x\in\gamma}|[x,(1-\lambda)w+\lambda(-v)]|$ is attained for a point of $\gamma_1$.  Analogously, if $\theta \in [0,1]$, then the supremum $\sup_{y\in\gamma}|[y,(1-\theta)w+\theta v]|$ is attained at some point of $\gamma_2$.
\end{lemma}
\begin{proof} It is clear that we just have to prove the first statement, since then the second follows from the duality explained in Lemma \ref{lemma21}. Let $\lambda \in [0,1]$. In view of Corollary \ref{coro16} it follows that the direction $(1-\lambda)w + \lambda(-v)$, which belongs to the second quadrant, supports $\mathrm{conv}(\gamma)$ at some point $\gamma_1(\theta)$. Hence, if $p \in \gamma\setminus\{-\gamma_1(\theta),\gamma_1(\theta)\}$, the assumption that the line $\left<(-p)p\right>$ intersects this supporting line at the point $\gamma_1(\theta)+t_0((1-\lambda)w+\lambda(-v))$ yields
\begin{align*} |[p,(1-\lambda)w+\lambda(-v)]| \leq \\
\leq |[\gamma_1(\theta) + t_0((1-\lambda)w +\lambda(-v)),(1-\lambda)w + \lambda(-v)]| = \\
= |[\gamma_1(\theta),(1-\lambda)w+\lambda(-v)]|, \end{align*}
and this shows what we wanted.

\end{proof}

\section{Radon curves as circles of Minkowski planes}\label{mink}

Now we want to prove that Birkhoff orthogonality is symmetric in a normed plane if and only if its unit circle is a Radon curve. The chief ingredient is the next lemma, but let us start with a definition:
Given a normed plane $(V,||\cdot||)$ endowed with a determinant form $[\cdot,\cdot]$, we define the \textit{antinorm} of a vector $x \in V$ to be
\begin{align*} ||x||_a := \sup\{|[y,x]|:y \in S\}. \end{align*}
It is not difficult to see that $||\cdot||_a$ is indeed a norm on $V$. The unit circle of $||\cdot||_a$ is called the \textit{anticircle} of $S$ and denoted by $S_a$. Moreover, the supremum is attained for $y \in S$ if and only if $y \dashv_B x$. This is, in some sense, the bridge that connects supporting relations (which come from Convex Geometry) with the construction of Radon curves.

\begin{lemma}\label{lemma19} Let $(V,||\cdot||)$ be a normed plane, and assume that $[\cdot,\cdot]$ is a fixed non-degenerate symplectic bilinear form, with associated antinorm $||\cdot||_a$.
Then the following statements are equivalent. \\

\noindent\normalfont\textbf{(a)} \textit{Birkhoff orthogonality is a symmetric relation in $(V,||\cdot||)$.} \\

\noindent\textbf{(b)} \textit{The unit anticircle and the unit circle are homothets.} \\

\noindent\textbf{(c)} \textit{There exists a number $\lambda > 0$ such that $||\cdot|| = \lambda||\cdot||_a$.}
\end{lemma}
\begin{proof} See \cite{martiniantinorms}.

\end{proof}

\begin{remark}\label{remark8}\normalfont Notice that in normed planes where the statements of Lemma \ref{lemma19} hold, the relation $||\cdot|| = \lambda||\cdot||_a$ gives, in some sense, a natural choice of symplectic forms (up to orientation): changing $[\cdot,\cdot]$ by $\lambda[\cdot,\cdot]$ it follows that $||\cdot|| = ||\cdot||_a$. In this case, the unit anticircle and the unit circle coincide.
\end{remark}

\begin{teo}\label{teo8} Let $(V,||\cdot||)$ be a normed plane. Then Birkhoff orthogonality is a symmetric relation if and only if the unit circle of the norm $||\cdot||$ is a Radon curve.
\end{teo}
\begin{proof} Assume first that the unit circle is a Radon curve $\gamma$ built as described previously (but possibly rescaling $[\cdot,\cdot]$ in order to have $[v,w] = 1$). Let $p \in \gamma$. Due to central symmetry and the duality described in Lemma \ref{lemma21} we may assume that $p = \gamma_1(\theta)$ for some $\theta \in [0,1]$. Hence, Lemma \ref{lemma23} yields
\begin{align*} ||\gamma_1(\theta)||_a = \sup_{y\in\gamma}|[y,\gamma_1(\theta)]| = \sup_{y\in\gamma_2}|[y,\gamma_1(\theta)]| = \\
=\sup_{y\in\gamma_2}\left|\left[y,\frac{(1-\theta)w+\theta v}{\sup_{z\in\gamma_2}|[z,(1-\theta)w+\theta v]|}\right]\right| = 1, \end{align*}
\noindent and therefore $||\cdot||_a = ||\cdot||$. For the converse, up to rescaling the symplectic form, assume that the antinorm and the norm coincide. Choose two conjugate diameters $\left<(-v)v\right>$ and $\left<(-w)w\right>$ and use the same notation as previously for the portions of $\gamma$ and quadrants determined by them. Any point $p \in \gamma_2$ can be written as $p =\alpha((1-\lambda)w + \lambda(-v))$ for some $\lambda \in [0,1]$ and some $\alpha >0$. Thus,
\begin{align*} 1 = ||p|| = ||p||_a = \sup_{x\in\gamma)}|[x,p]| = \sup_{x\in\gamma}|[x,\alpha((1-\lambda)w + \lambda(-v)]|. \end{align*}
It follows that $\alpha = \left(\sup_{x\in\gamma}|[x,(1-\lambda)w+\lambda(-v)]|\right)^{-1}$. Then, to show that the unit circle is a Radon curve, we just have to prove that this supremum is attained for some point of $\gamma_1$. For this sake, it is enough to repeat the proof of Lemma \ref{lemma23} using Proposition \ref{prop25} instead of Corollary \ref{coro16}.

\end{proof}

It is clear that the choice of a non-degenerate symplectic bilinear form gives an area measure. We finish this section with a characterization of Radon planes which, geometrically, means that any rectangle (in the Birkhoff sense) with unit sides has the same area if and only if the norm is Radon.

\begin{prop}\label{prop24} A normed plane $(V,||\cdot||)$ is Radon if and only if there exists a number $\lambda > 0$ such that $|[x,y]| = \lambda$ whenever $x$ and $y$ are unit vectors with $x \dashv_B y$.
\end{prop}
\begin{proof} If $||\cdot||$ is a Radon norm, then there exists a number $\lambda > 0$ such that $||\cdot||_a = \lambda||\cdot||$. Hence, if $x$ and $y$ are unit vectors such that $x\dashv_B y$, then 
\begin{align*} |[x,y]| = \sup_{z\in S}|[z,y]| = ||y||_a = \lambda. \end{align*}
\noindent Now, if $(V,||\cdot||)$ is not a Radon plane, we may choose vectors $x,y \in S$ such that $x$ is orthogonal to $y$, but  the converse is not true. Hence we may choose $z\in S\setminus\{y\}$ with $z \dashv_B x$,
and it follows that
\begin{align*}
|[z,x]| = \sup_{w\in S}|[w,x]| > |[y,x]|\,.
\end{align*}

\noindent This finishes our proof.

\end{proof}

\section{Further comments}\label{comments}

The existence of non-Euclidean norms for which Birkhoff orthogonality is a symmetric relation is a two-dimensional phenomenon. Indeed, if $(V,\|\cdot\|)$ is a Minkowski space with $\mathrm{dim}V \geq 3$, then a norm on it has symmetric Birkhoff orthogonality if and only if it is derived from an inner product (see Theorem 3.4.10 in \cite{thompson}).
Radon planes behave like the Euclidean plane regarding many properties. Also, there are many nice characterizations of Radon planes among all normed planes. For results in this direction we refer the reader to the papers \cite{bmt}, \cite{duvelmeyer}, \cite{martiniantinorms}, and \cite{martini1}, and to \S~4.7 and \S~4.8 in \cite{thompson}.
Some of these results can be described only in terms of Convex Geometry. We present two examples. The first one is merely a rewriting of Proposition \ref{prop24}.

\begin{prop} Let $\gamma$ be a closed curve which is the boundary of a convex body in a plane $V$ (endowed with a determinant form $[\cdot,\cdot]$) and centered at the origin.
Then $\gamma$ is a Radon curve if and only if there exists a number $\lambda > 0$ such that $|[x,y]| = \lambda$ whenever $x,y \in \gamma$ are such that the direction $y$ supports $\mathrm{conv}(\gamma)$ at $x$.
\end{prop}

In \cite{duvelmeyer}, D\"{u}velmeyer proved that a norm is Radon if and only if Busemann and Glogovskii angular bisectors coincide for any angle (definitions are given in the proof below).
This yields immediately the following non-Minkowskian characterization of Radon curves.

\begin{prop}\label{propchar} Let $\gamma$ be a closed  curve in a plane $V$ which is the boundary of a convex body and centered at the origin. Then $\gamma$ is a Radon curve if and only if for every $p \in V\setminus\mathrm{conv}(\gamma)$ the following property holds: let $r$ and $s$ be the tangents to $\mathrm{conv}(\gamma)$ passing through $p$. Let $x_0$ and $x_1$ be the points where the line parallel to $r$ through the origin intersects $\gamma$ and $s$, respectively, and let $y_0$ and $y_1$ be the respective intersections of the line parallel to $s$ and passing through the origin with $\gamma$ and $r$ (see Figure \ref{fig1r}).
Then the line through $x_1$ and $y_1$ is parallel to the line through $x_0$ and $y_0$.
\end{prop}
\begin{proof} First, notice that every angle can be realized, up to translation, as the angle formed by two concurrent tangent lines to $\gamma$. It is known that given a point $p \in V\setminus\mathrm{conv}(\gamma)$ and lines $r$ and $s$ as in the enunciate, the \textit{Glogovskii angular bisector} of the angle determined by $r$ and $s$ is the line $\left<op\right>$. (In the language of Minkowski planes, the Glogovskii bisector of the angle determined by $r$ and $s$ consists of all midpoints of norm circles having the rays of  this angle in tangential position.)
 On the other hand, the \textit{Busemann angular bisector} of this angle is the ray starting at $p$ in the direction of the sum of the unit vectors (with respect to the norm having $\gamma$ as unit circle) in directions $r$ and $s$. This can easily be formulated not depending on norms, and the desired follows.

\begin{figure}[h]
\centering
\includegraphics{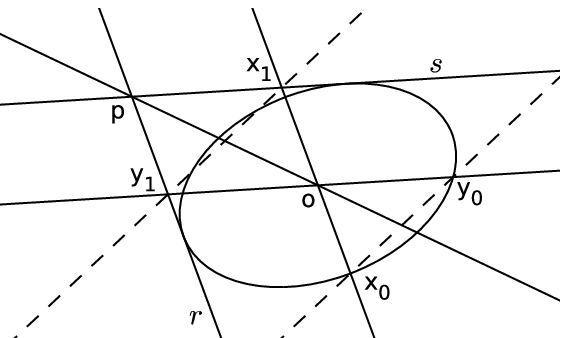}
\caption{Proposition \ref{propchar}}
\label{fig1r}
\end{figure}

\end{proof}

\begin{remark}\normalfont We underline once more that this characterization of Radon curves relies only in basic concepts of vectorial spaces. We even need not fix a determinant form.
\end{remark}

\noindent{\bf Acknowledgements} The first named author thanks to CAPES for partial financial support during the preparation of this manuscript.

\end{document}